\documentclass[12pt]{amsart}
\usepackage[headings]{fullpage}

\usepackage{amsmath,amssymb,amsthm}

\usepackage{graphicx}
\usepackage{enumerate}
\usepackage{xcolor}
\usepackage{url}
\usepackage{slashed}
\usepackage{bm}
\usepackage{tikz-cd}
\usepackage{pb-diagram}
\usepackage[german,english]{babel}

\usepackage{enumitem}
\usepackage[bookmarks=true,%
    colorlinks=true,%
    linkcolor=blue,%
    citecolor=blue,%
    filecolor=blue,%
    menucolor=blue,%
    urlcolor=blue,%
    breaklinks=true]{hyperref}

\usepackage[all]{xy}
\usepackage{verbatim}

\newtheorem{thm}{Theorem}[]
\newtheorem*{thm*}{Theorem}
\newtheorem{lem}[thm]{Lemma}
\newtheorem{prop}[thm]{Proposition}
\newtheorem{cor}[thm]{Corollary}

\newtheorem{rem}[thm]{Remark}
\newtheorem{defn}[thm]{Definition}

\newcommand{\param}{{\mathchoice{\mkern1mu\mbox{\raise2.2pt\hbox{$
\centerdot$}}
\mkern1mu}{\mkern1mu\mbox{\raise2.2pt\hbox{$\centerdot$}}\mkern1mu}{
\mkern1.5mu\centerdot\mkern1.5mu}{\mkern1.5mu\centerdot\mkern1.5mu}}}

\renewcommand{\setminus}{{\smallsetminus}}

\begin{document}
\title{An equivariant deformation retraction of the Thurston spine}
\author{Ingrid Irmer}
\address{SUSTech International Center for Mathematics\\
Southern University of Science and Technology\\Shenzhen, China
}
\address{Department of Mathematics\\
Southern University of Science and Technology\\Shenzhen, China
}
\email{ingridmary@sustech.edu.cn}


\begin{abstract}
This paper shows that there is a mapping class group-equivariant deformation retraction of the Teichm\"uller space of a closed, orientable surface onto a cell complex of dimension equal to the virtual cohomological dimension of the mapping class group. The image of the deformation retraction is contained in the CW complex first described by Thurston -- the Thurston spine. The Thurston spine is the set of points in Teichm\"uller space corresponding to hyperbolic surfaces for which the set of shortest geodesics (the systoles) cuts the surface into polygons. 
\end{abstract}

\maketitle

{\footnotesize
\tableofcontents
}

\section{Introduction}
\label{sect:intro}

In this paper it will be shown that the Teichm\"uller space of a closed, orientable surface of genus $g\geq 2$ has a mapping class group-equivariant deformation retraction onto a CW complex of dimension equal to the virtual cohomological dimension of the mapping class group, namely $4g-5$, see \cite{Harer}. \\

In genus 2, the existence of a mapping class group-equivariant deformation retraction onto a CW complex of dimension 3 follows from a computation by Schmutz Schaller, see Theorem 44 of \cite{SchmutzMorse}. This also follows from a SageMath calculation \cite{Calculation}, using Rivin's angle coordinates on a sphere with six cone points. When the genus is greater than or equal to 2 and the surface has at least one puncture, a mapping class group-equivariant deformation retraction to a CW complex of dimension equal to the virtual cohomological dimension of the mapping class group was given in \cite{Harer} and \cite{PennerComplex}. This result has had a number of applications; for example, it was used in \cite{Kontsevich} to prove a conjecture of Witten's about intersection theory on the moduli space of Riemann surfaces with punctures. In the absence of punctures, determining whether a deformation retraction onto a CW complex of dimension equal to the virtual cohomological dimension actually exists is listed as the first open question in \cite{BV}. The interested reader is referred to \cite{Bestvina2}, \cite{BV} and Chapter 3.3 of \cite{Ha} for a survey of the background and applications of this question. \\

Thurston's construction in \cite{Thurston} (see \cite{Ha}, \cite{Me2023}, \cite{MorseSmale}, \cite{PT} for more details) resolved the problem of a missing basepoint by using length functions to parametrise Teichm\"uller space, constructing a mapping class group-equivariant deformation retraction of the Teichm\"uller space of a closed, compact surface. The image of this deformation retraction is the so-called \textit{Thurston spine} $\mathcal{P}_{g}$. This is a CW-complex contained in $\mathcal{T}_{g}$ consisting of the set of points representing hyperbolic surfaces that are cut into polygons by the set of shortest geodesics (also known as the systoles). \\

\begin{thm}
\label{dimensionthm}
There is a mapping class group-equivariant deformation retraction of the Thurston spine of a closed orientable surface of genus $g$ onto a CW-complex of dimension equal to $4g-5$.
\end{thm}

Together with the construction from \cite{Thurston} this gives 

\begin{thm}
\label{dimensionthm2}
There is a mapping class group-equivariant deformation retraction of the Teichm\"uller space of a closed orientable surface of genus $g$ onto a CW complex of dimension equal to $4g-5$.
\end{thm}

For large genus, the reduction in dimension achieved by the deformation retraction from Theorem \ref{dimensionthm} can be significant, as there are examples known for which the codimension of the Thurston spine is small relative to the genus of the surface. This follows from the construction in \cite{Maxime} of small index critical points of the topological Morse function $f_{\mathrm{sys}}:\mathcal{T}_{g}\rightarrow \mathbb{R}$, mapping a point $x$ of $\mathcal{T}_{g}$ to the length of the systoles at $x$. Any spine constructed as in \cite{Thurston} must contain the ``unstable manifolds'' of critical points of  the systole function $f_{\mathrm{sys}}$, so the index of a critical point gives an upper bound on the codimension of $\mathcal{P}_{g}$. This is explained in Section 5 of \cite{MorseSmale}.\\

The Teichm\"uller space $\mathcal{T}_{g}$ is contractible, and by Fricke's theorem, the mapping class group acts properly discontinuously on it. Studying mapping class group-equivariant deformation retractions of Teichm\"uller space is intimately connected with questions about the virtual cohomological dimension of the mapping class group and about the problem of finding a space of the lowest possible dimension on which the mapping class group acts properly discontinuously. The virtual cohomological dimension gives a lower bound on this dimension. Theorem \ref{dimensionthm2} shows that this lower bound is achieved.\\

\textbf{The Steinberg module and the Thurston Spine.}
The ``thick'' part, $\mathcal{T}_{g}^{\epsilon_{M}}$, of $\mathcal{T}_{g}$ is defined to be the set of all points of $\mathcal{T}_{g}$ corresponding to surfaces with injectivity radius greater than or equal to the Margulis constant, denoted by $\epsilon_{M}\simeq 0.2629$ \cite{MargulisConstant}. The Margulis constant has many important geometric and algebraic properties; for example, when $f_{\mathrm{sys}}$ is less than or equal to $\epsilon_{M}$, it follows from the Margulis lemma that the systoles are pairwise disjoint. It was shown in \cite{Harer} that Harvey's complex of curves $\mathcal{C}_{g}$ is homotopy equivalent to an infinite wedge of spheres of dimension $2g-2$. It is known, \cite{Ivanov2002}, that $\partial\mathcal{T}_{g}^{\epsilon_{M}}$ is mapping class group-equivariantly homotopy equivalent to $\mathcal{C}_{g}$.\\

The Steinberg module of the mapping class group is defined to be the homology group $H_{2g-2}\bigl(\mathcal{C}_{g};\mathbb{Z}\bigr)$. As there is a simplicial action of the mapping class group on the complex of curves, the Steinberg module inherits the structure of a mapping class group-module. \\

Theorem \ref{dimensionthm} is proven by using the homology of $\partial \mathcal{T}_{g}^{\epsilon_{M}}$ to show that certain cells dual to the spine --- essentially the cells defined by Schmutz Schaller in \cite{SchmutzMorse} --- must be homotopic into $\partial \mathcal{T}_{g}^{\epsilon_{M}}$ relative to their intersection with $\partial \mathcal{T}_{g}^{\epsilon_{M}}$. When the dimension of $\mathcal{P}_{g}$ is greater than $4g-5$, this shows the existence of a boundary of $\mathcal{P}_{g}$, used as a starting point for a deformation retraction. This deformation retraction can be done equivariantly, because the mapping class group preserves the level sets of $f_{\mathrm{sys}}$, and the intersection of the fixed point sets of the mapping class group with $\mathcal{P}_{g}$ are critical points or are transverse to the level sets of $f_{\mathrm{sys}}$.\\

\subsection*{Acknowledgements} 
The author would like to thank Stavros Garoufalidis, Xiaolong Han, Scott Wolpert, Huang Yi and Don Zagier for commenting on myriad earlier drafts of this paper which led to considerable improvements. 
\section{Definitions and conventions}
\label{sub:defns}

This section provides definitions and background used in the rest of the paper. \\


The Teichm\"uller space of a closed, compact, connected topological surface $\mathcal{S}_{g}$ of genus $g\geq 2$ will be denoted by $\mathcal{T}_{g}$, and the mapping class group of $\mathcal{S}_{g}$ by $\Gamma_{g}$. Moduli space $\mathcal{M}_{g}$ is the quotient of $\mathcal{T}_{g}$ by this action.\\

Curves on surfaces are all assumed to be simple, closed, isotopy classes of maps of $S^{1}$ onto the surface. For convenience, the word curve will also refer to the image of a particular representative of the isotopy class such as a geodesic. Curves on a marked hyperbolic surface representing a point in $\mathcal{T}_{g}$ are also assumed to inherit a marking. The length of a curve on a hyperbolic surface is the length of the geodesic representative of the isotopy class. A curve will usually be denoted by a lowercase $c$, perhaps with a subscript, while a finite set of curves will be denoted by an uppercase $C$.\\

A set of curves fills the surface if the complement of the geodesic representatives is a union of polygons. The Thurston spine, $\mathcal{P}_{g}$, is the subset of $\mathcal{P}_{g}$ on which the systoles fill. \\

The systole function $f_{\mathrm{sys}}:\mathcal{T}_{g}\rightarrow \mathbb{R}^{+}$ is the function whose value at $x$ is give by the length of the systoles. A summary of properties of the systole function is as follows:
\begin{enumerate}
\item{It is bounded from below by zero, and from above by Bers' constant, \cite{Bers1985}. Bers' constant depends only on $g$, and is bounded from above by $21(g-1)$, \cite{FandM}.}
\item{It is continuous, and smooth when restricted to a stratum (defined later).}
\item{It is a topological Morse function (defined later).}
\end{enumerate}

The Thurston spine is nonempty as a result of property (1) above and Proposition \ref{Thurstonprop} below. \\

\textbf{Local finiteness.} The set of curves of length less than a given real number is finite, \cite{M}. For any point $x\in \mathcal{T}_{g}$ there is an open neighbourhood of $x$ in which the systoles of $x'$ in the neighbourhood are contained in the set of systoles at $x$. \\

Harvey's curve complex, \cite{Harvey}, is the simplicial complex with $n$-simplices in 1-1 correspondence with sets of $n+1$ homotopically nontrivial, pairwise disjoint, closed curves on $\mathcal{S}_{g}$. Since curves are defined up to isotopy, ``pairwise disjoint'' means geometric intersection number zero. Due to the fact that the action of $\Gamma_{g}$ on curves preserves intersection properties, there is a simplical action of $\Gamma_{g}$ on $\mathcal{C}_{g}$.\\ 

The Weil-Petersson metric on $\mathcal{T}_{g}$ will be assumed when defining gradients, etc. However, in most cases, any other equivariant metric could also be used.\\

\begin{defn}[Length function]
A curve $c$ on $\mathcal{S}_{g}$ determines an analytic function $L(c):\mathcal{T}_{g}\rightarrow \mathbb{R}^{+}$ whose value at a point $x\in \mathcal{T}_{g}$ is equal to the length of $c$ at $x$. For a finite ordered set of curves $C=(c_{1}, \ldots, c_{k})$ and an ordered set of real, positive weights $A=(a_{1}, \ldots, a_{k})$, a length function $L(A,C):\mathcal{T}_{g}\rightarrow \mathbb{R}^{+}$ is a function given by
\begin{equation*}
L(A,C)(x)\,=\, \sum_{j=1}^{k} a_{j}L(c_{j})(x)
\end{equation*}
\end{defn}

A survey of convexity properties of length functions is given in \cite{Bestvina}, for example, they are convex along earthquake paths, strictly convex on Weil-Petersson geodesics and Fenchel-Nielsen coordinates can be chosen such that length functions are convex functions of these coordinates. Another important property of length functions is given by the following.\\

\begin{prop}[Proposition 0.1 of \cite{Thurston}, Lemma 4 of \cite{Bers}]
\label{Thurstonprop}
Suppose $C$ is a set of curves on a surface that do not fill. Then at any point $x$ of $\mathcal{T}_{g}$, there is an open cone of derivations in $T_{x}\mathcal{T}_{g}$ whose evaluation on every $L(c)$ for $c\in C$ is strictly positive. 
\end{prop}

There is a stratification of $\mathcal{T}_{g}$, where a stratum, $\mathrm{Sys}(C)$, is the subset of $\mathcal{T}_{g}$ on which $C$ is the set of systoles for some fixed finite set $C$. Every point lies in a unique stratum, and there is an action of the mapping class group on the strata, with finitely many orbits, \cite{Thurston}. The function $f_{\mathrm{sys}}$ is smooth when restricted to each stratum in $\mathcal{T}_{g}$, but at the boundary of two different strata, one-sided limits of the derivative do not usually match. \\




As $f_{\mathrm{sys}}$ is not smooth, it is not a Morse function, but has many properties in common with a Morse function.

\begin{defn}[Topological Morse function]
Let $M$ be an $n$-dimensional topological manifold, and $f:M\rightarrow \mathbb{R}^{+}$ a continuous function. Then $f$ is a topological Morse function if the points of $M$ consist of regular points and critical points. A regular point $p \in M$ is characterised by the property of having an open neighbourhood $U$, where $U$ admits a homeomorphic parametrisation by $n$ parameters, one of which is $f$. When $p$ is a critical point, there exists a $k\in \mathbb{Z}$, $0\leq k\leq n$, referred to as the index of $p$, and a homeomorphic parametrisation of $U$ by parameters $\{x_{1}, \ldots, x_{n}\}$, such that at every point in $U$, $f$ satisfies
\begin{equation*}
f(x)\,-\,f(p) \;=\ \sum_{i=1}^{i=n-k}x_{i}^{2}\,-\,\sum_{i=n-k+1}^{i=n}x_{i}^{2}
\end{equation*}
\end{defn}

Topological Morse functions were first defined in \cite{Morse}, where it was shown that, when they exist, they can be used in most of the same ways as their smooth analogues for constructing cell decompositions of manifolds and computing homology. Another reference, in which the standard theorems (except existence and deformability of Morse functions) are shown to generalise from the smooth to the topological category is given by Section 3 of Essay III of \cite{Siebenmann}. In \cite{Akrout} it was shown that $f_{\mathrm{sys}}$ is a topological Morse function, proving a conjecture from \cite{SchmutzMorse}.\\


It is a consequence of Lojasiewicz's theorem, \cite{Lojasiewicz1964} that $\mathcal{P}_{g}$ admits a triangulation compatible with the stratification. This means that the interior of any simplex is constained in a single stratum. It follows that $\mathcal{P}_{g}$ has a tangent space almost everywhere. In this paper, $\mathcal{P}_{g}$ will be described as a CW complex or a simplicial complex, depending on what is most convenient at the time. From now on, a fixed triangulation will be assumed.\\

\textbf{Tangent cone and cone of increase.} For a point $p$ in a simplex $T$ in $\mathcal{P}_{g}$, the \textit{tangent cone} to $T$ at $p$ is the set of $v\in T_{p}\mathcal{T}_{g}$ such that $v=\dot{\gamma}(0)$ for a smooth oriented path $\gamma(t)$ with $\gamma(0)=p$ and $\gamma(\epsilon)$ in $T$ for sufficiently small $\epsilon>0$. If $p\in \mathcal{P}_{g}$ is on the boundary of more than one simplex, the tangent cone to $\mathcal{P}_{g}$ at $p$ is the union of the tangent cones of the simplices with $p$ on the boundary. It is possible to define tangent cones to strata by taking a triangulation compatible with the stratification, in which case the tangent cone to the stratum at $p$ is the union of the tangent cones of the simplices with interior contained in the stratum and with $p$ on the boundary. Lojasiewicz's theorem also applies to level sets within strata, which can therefore also be triangulated, and hence have tangent cones. The \textit{unit tangent cone} is the set of vectors in the tangent cone with unit length. Tangent cones were defined similarly in \cite{Thurstonstretch}. \\

The systole function $f_{\mathrm{sys}}$ is only piecewise-smooth, so instead of a gradient it has a cone of increase. By local finiteness, at a point $p\in \mathrm{Sys}(C)$, $f_{\mathrm{sys}}$ is increasing or stationary in a direction iff the lengths of all the curves in $C$ are increasing or stationary in that direction. For $p\in \mathrm{Sys}(C)$, the cone of increase of $f_{\mathrm{sys}}$ at $p$ is defined to be the tangent cone at $p$ of the intersection $$I(C,p):=\{x\in \mathcal{T}_{g}\ |\ L(c)(x)\geq f_{\mathrm{sys}}(p)\ \forall c\in C\}$$. The cone of increase is full at $p$ if it has dimension at $p$ equal to the dimension of $T_{p}\mathcal{T}_{g}$. \\
 
\textbf{Boundary points and critical points of }$\mathbf{f}_{\mathbf{\mathrm{sys}}}$.  On a neighbourhood of a regular point of $f_{\mathrm{sys}}$, the level sets are only homeomorphic to codimension 1 submanifolds. As pointed out in \cite{SchmutzMorse}, it is theoretically possible that there exist so-called \textit{boundary points} of $f_{\mathrm{sys}}$; regular points at which the cone of increase has dimension less than that of $\mathcal{T}_{g}$. No examples  of these are known. As shown in \cite{SchmutzMorse}, both critical points and boundary points of $f_{\mathrm{sys}}$ are isolated. Moreover, the characterisation of boundary points of $f_{\mathrm{sys}}$ in \cite{SchmutzMorse} implies that all boundary points are contained in $\mathcal{P}_{g}$. By Proposition \ref{Thurstonprop}, the same is true of critical points of $f_{\mathrm{sys}}$.\\




\begin{rem}[Interpreting figures]
As it is not possible to draw objects in 6 dimensions and more, figures in this paper should be interpreted as illustrations of lower dimensional analogues of the objects defined and phenomena under discussion.
\end{rem}

\begin{figure}[h!]
\centering
\includegraphics[width=0.6\textwidth]{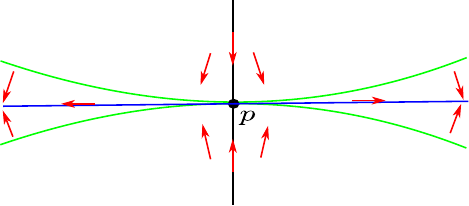}
\caption{The black line represents the subspace of $T_{p}\mathcal{T}_{g}$ spanned by $\{\nabla L(c)\ |\ c\in C\}$. The level sets of $f_{\mathrm{sys}}$ passing through the critical point $p$ are drawn in green. The arrows show the direction in which $f_{\mathrm{sys}}$ is increasing. The simplices $\sigma_{1}$ and $\sigma_{2}$ in $\mathcal{P}_{g}$ guaranteed by Corollary \ref{cor} are shown in blue.}
\label{localstructurefig}
\end{figure}

\begin{cor}
\label{cor}
At critical points and boundary points of $f_{\mathrm{sys}}$, the cone of increase of $f_{\mathrm{sys}}$ is contained in the tangent cone to $\mathcal{P}_{g}$. At a critical point $p$ with set of systoles $C$, the cone of increase of $f_{\mathrm{sys}}$ is given by $\{\nabla L(c)(p)\ |\ c\in C\}^{\perp}$. 
\end{cor}

Note that $\{\nabla L(c)(p)\ |\ c\in C\}^{\perp}$ is independent of the choice of metric, as it can be described as the tangent cone to intersections of level sets of $\{ L(c)(p)\ |\ c\in C\}$.\\

\begin{proof}
As discussed in Section 5 of \cite{MorseSmale}, this is a consequence of the model for critical points and boundary points of $f_{\mathrm{sys}}$ resulting from the work of Schmutz Schaller and Akrout. 
\end{proof}

By local finiteness and corollary \ref{cor}, at a critical point, the cone of increase only contains directions in which $f_{\mathrm{sys}}$ is increasing to second or higher order. On a neighbourhood of a critical point of $f_{\mathrm{sys}}$, the level sets are homeomorphic to the level sets on a neighbourhood of a critical point of a smooth Morse function of the same index. This is shown in Figure \ref{localstructurefig} for a toy example with two systoles. 

\begin{figure}[!thpb]
\centering
\includegraphics[width=0.2\textwidth]{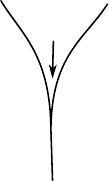}
\caption{Where $\mathcal{P}_{g}$ splits apart, Theorem 2 of \cite{MorseSmale} implies $f_{\mathrm{sys}}$ is increasing towards the split. The arrow shows the direction of increase of $f_{\mathrm{sys}}$.}
\label{ruledout}
\end{figure}

\begin{thm}[Theorem 1 of \cite{MorseSmale}]
\label{thm1}
Let $p\in \mathcal{P}_{g}$ be a point at which the cone of increase of $f_{\mathrm{sys}}$ is full (i.e. $p$ is not a critical or boundary point of $f_{\mathrm{sys}}$). Then the intersection of the cone of increase at $p$ with the tangent cone of $\mathcal{P}_{g}$ at $p$ is nonempty.
\end{thm}

A point $x\in \mathcal{P}_{g}$ is defined to be balanced if $x$ is either a critical point or there is a direction in the cone of increase at $x$ in the convex hull of the gradients of the lengths of the systoles.

\begin{thm}[Theorem 2 of \cite{MorseSmale}]
\label{thm2}
Where a stratum of $\mathcal{P}_{g}$ is locally top-dimensional it is balanced. If $x\in \mathcal{P}_{g}$ is not balanced, there are points directly above $x$ in $\mathcal{P}_{g}$ that are balanced.
\end{thm}

Theorem \ref{thm2} is used to describe the behaviour of $f_{\mathrm{sys}}$ on a neighbourhood of $\mathcal{P}_{g}$, and was shown to hold independently of the choice of equivariant metric. The purpose of this theorem is to capture Thurston's intuition that $f_{\mathrm{sys}}$ is increasing towards $\mathcal{P}_{g}$. At every point at which $\mathcal{P}_{g}$ has a tangent space, with respect to every choice of $\Gamma_{g}$-equivariant metric, this theorem implies that $f_{\mathrm{sys}}$ is decreasing or stationary in an outward-pointing normal direction to $\mathcal{P}_{g}$. Informally speaking, this rules out the possibility of having a disconnected set of directions in the tangent cone of $\mathcal{P}_g$ in which $f_{\mathrm{sys}}$ is increasing, as illustrated in Figure \ref{ruledout}. \\

The cohomological dimension of a group $G$ is
\begin{equation*}
\sup\bigl\{n\in \mathbb{N} \,|\,H^{n}(G, M)\neq 0 \text{ for some module } M\bigr\}.
\end{equation*}

The mapping class group is known to contain finite index torsion-free subgroups; a discussion is given in Chapter 6 of \cite{FandM}. Cohomological dimension is not an exciting invariant of groups with torsion, as it is known that it is always infinite in this case, \cite{Brown}. By Serre's theorem, \cite{Brown} Chapter VIII, any finite index torsion-free subgroup of a group has the same cohomological dimension. The cohomological dimension of a (and hence any) finite index torsion free subgroup is then called the virtual cohomological dimension. In the case of $\Gamma_{g}$, the virtual cohomological dimension was shown in \cite{Harer} to be $4g-5$.

\section{The deformation retraction of the Thurston spine}
\label{dimensioncalculation}

This section shows that there is a $\Gamma_{g}$-equivariant deformation retraction of the Thurston spine onto a CW-complex of dimension equal to the virtual cohomological dimension of $\Gamma_{g}$. The exposition begins with a brief survey of properties of fixed point sets of the action of subgroups of $\Gamma_{g}$ on $\mathcal{T}_{g}$ in Subsection \ref{sub}. Following this, an equivariant deformation retraction of $\mathcal{P}_{g}$ onto a complex $\mathcal{P}_{g}^{X}$ is constructed in Subsection \ref{pgx} and a definition of stable and unstable manifolds is given. Conceptually, this is a familiar construction from differential topology, however there are a few technical issues discussed in detail that arise from the fact that the construction is performed on a simplical complex rather than on a manifold, and the Morse function is not smooth. One of these is the possibility of boundary points of $f_{\mathrm{sys}}$. If boundary points of $f_{\mathrm{sys}}$ exist, this makes it necessary to alter the vector field $X$, as explained in Subsection \ref{boundarypoints}. In Subsection \ref{unmatchedfaces}, the topology of $\partial\mathcal{T}_{g}^{\epsilon_{M}}$ is used to show that when $\mathcal{P}_{g}^{X}$ has dimension greater than $4g-5$, $\mathcal{P}_{g}^{X}$ necessarily has a boundary. This boundary is used as the starting point of an equivariant deformation retraction. As explained in Subsection \ref{construction}, this deformation retraction is obtained as a composition of orbits of simpler deformation retractions, defined using the structure of unstable manifolds and the properties of flowlines and fixed point sets.\\


\subsection{Fixed point sets}
\label{sub}

By Fricke's theorem, $\Gamma_{g}$ acts properly discontinuously on $\mathcal{T}_{g}$. The stabiliser of any point of $\mathcal{T}_{g}$ is therefore finite. Moreover, it is known that every finite subgroup of $\Gamma_{g}$ has a nonempty fixed point set, \cite{NRP}. A fixed point set of a finite subgroup of $\Gamma_{g}$ is the set of all points of $\mathcal{T}_{g}$ corresponding to hyperbolic surfaces with isometry group given by the subgroup, i.e. the subgroup fixes the set pointwise.\\

Due to the fact that $\Gamma_{g}$ acts by isometry with respect to a number of metrics, such as the Teichm\"uller metric and the Weil-Petersson metric, each connected component of a fixed point set of a finite subgroup of $\Gamma_{g}$ is a closed, totally geodesic submanifold with respect to these metrics. \\

\begin{prop}
\label{fixedpointsetprop}
Zero dimensional fixed point sets of finite subgroups of $\Gamma_{g}$ acting on $\mathcal{T}_{g}$ are critical points of every $\Gamma_{g}$-equivariant Morse function on $\mathcal{T}_{g}$. Moreover, away from the critical points, the fixed point sets of finite subgroups of $\Gamma_{g}$ intersect the level sets of every $\Gamma_{g}$-equivariant Morse function transversely.
\end{prop}

The level sets of $f_{\mathrm{sys}}$ are not smooth. The notion of transversality used in Proposition \ref{fixedpointsetprop} is given in \cite{Armstrong}. 

\begin{proof}
If a 0-dimensional fixed point set were not a critical point, since the action of $\Gamma_{g}$ maps the unit cone of increase of the Morse function to itself, Brouwer's fixed point theorem implies the existence of a vector that is fixed by the action of the subgroup. A Weil-Petersson geodesic through $p$ with this tangent vector is therefore also contained in the fixed point set, contradicting the assumption that it is 0-dimensional.\\

Suppose $p$ is contained in the fixed point set of a finite subgroup of $\Gamma_{g}$, and $p$ is not a critical point. Since the finite index subgroup maps the unit cone of increase of the Morse function at $p$ to itself, as argued above, there is a vector in the cone of increase of the Morse function at $p$ that is fixed by the action of the subgroup. A Weil-Petersson geodesic through $p$ with this tangent vector is therefore also contained in the fixed point set. This geodesic is transverse to the codimension 1 level sets of the Morse function.
\end{proof}




\subsection{Extending Thurston's flow}
\label{pgx}

In \cite{Morse}, stable and unstable manifolds of a topological Morse function were defined. In contrast to the case for Morse functions, the stable and unstable manifolds of critical points of $f_{\mathrm{sys}}$, although defined, may not be uniquely defined. Nevertheless, since $\mathcal{P}_{g}$ contains all the critical points, and is fixed by Thurston's $f_{\mathrm{sys}}$-increasing flow, it also $\Gamma_{g}$-equivariantly deformation retracts on a union $\mathcal{P}_{g}^{X}$ of some notion of the unstable manifolds of the critical points. Here $X$ is an $f_{\mathrm{sys}}$-increasing vector field defined below. The purpose of this subsection is to construct $\mathcal{P}_{g}^{X}$ and use it to define unstable manifolds.\\

In this subsection, it will be assumed that there are no boundary points of $f_{\mathrm{sys}}$. In subsection \ref{boundarypoints} it will be explained how to generalise the results of this subsection in the presence of boundary points of $f_{\mathrm{sys}}$, by altering $X$ in such a way that the only zeros occur at critical points.\\

Theorem \ref{thm1} will first be used to perform a $\Gamma_{g}$-equivariant deformation retraction of $\mathcal{P}_{g}$ onto a CW-complex $\mathcal{P}_{g}^{X}$, using the flow of a vector field $X$ on $\mathcal{P}_{g}$. This deformation retraction is the largest $\Gamma_{g}$-equivariant deformation retraction possible that does not cancel critical points, and it is constructed in such a way as to preserve the level set structure of $f_{\mathrm{sys}}$. By Theorem \ref{thm1}, the restriction of $f_{\mathrm{sys}}$ to $\mathcal{P}_{g}$ only has stationary points at critical and boundary points of $f_{\mathrm{sys}}$. Around a critical point, the local behaviour of $f_{\mathrm{sys}}$ is given by Corollary \ref{cor}.  \\

Given two points $x, x'$ of $\mathcal{T}_{g}$, $x'$ will be said to be \textit{above} $x$ if $f_{\mathrm{sys}}(x')>f_{\mathrm{sys}}(x)$. Whenever $x$ is on the boundary of a simplex $\sigma$ of $\mathcal{P}_{g}$, $\sigma$ will be said to be \textit{above} $x$ if the cone of increase of $f_{\mathrm{sys}}$ at $x$ has nonempty intersection with the tangent cone of $\sigma$ at $x$.\\

Assume a triangulation of $\mathcal{P}_{g}$ compatible with the stratification. A $\Gamma_{g}$-equivariant vector field $X$ on $\mathcal{P}_{g}$ will now be defined. The properties of $X$ will first be discussed, then listed, after that it will be proven that a well-defined vector field satisfying these conditions exists. The term flowline will be used to denote a piecewise-smooth path in $\mathcal{P}_{g}$, for which the vector field $X$ is the 1-sided limit from above of the tangent vector to the path at every point.\\

\begin{figure}[h!]
\centering
\includegraphics[width=0.3\textwidth]{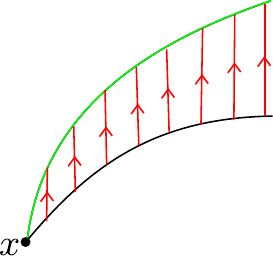}
\caption{At the points shown in green, the vector field $X$ is not determined by the vector field on the adjacent higher dimensional stratum, where the flowlines are shown in red.}
\label{sichel}
\end{figure}

By Theorem \ref{thm1}, the only zeros of $X$ on $\mathcal{P}_{g}$ occur at critical points and boundary points of $f_{\mathrm{sys}}$. At every other point $x\in\mathcal{P}_{g}$, $X(x)$ is in the tangent cone to $\mathcal{P}_{g}$ and is a direction in which $f_{\mathrm{sys}}$ is increasing. To begin with, it will be assumed that in the interior of every locally top dimensional simplex $\sigma$, $X$ is given by the gradient of $f_{\mathrm{sys}}$ restricted to $\sigma$. This condition completely determines the flowlines on $\sigma$. It also determines $X$ at every point $x$ in a lower dimensional simplex that is the endpoint of a segment of a flowline above $x$. Now suppose $x\in \mathcal{P}_{g}$ is not such a limit point, see for example Figure \ref{sichel}. If moreover $x$ is contained in the interior of a simplex $\sigma_{1}$ of the largest dimension with this property, $X(x)$ is defined to be the gradient of the restriction of $f_{\mathrm{sys}}$ to $\sigma_{1}$.\\



If necessary, $X$ can be scaled by a positive $\Gamma_{g}$-equivariant function. \\

\begin{lem}
\label{list}
When $f_{\mathrm{sys}}$ has no boundary points, there is a well-defined vector field $X$ on $\mathcal{P}_{g}$ satisfying the following properties.
\begin{enumerate}
\item $f_{\mathrm{sys}}$ is everywhere increasing in the direction of $X$, except at critical points of $f_{\mathrm{sys}}$, where $X=0$. In the interior of every locally top dimensional simplex, as well as some other points as explained above, $X$ is parallel to the gradient of the restriction of $f_{\mathrm{sys}}$ to the simplex.
\item $X$ is $\Gamma_{g}$-equivariant.
\item At every point $x\in \mathcal{P}_{g}$, $X(x)$ is in the tangent cone to $\mathcal{P}_{g}$ at $x$.
\item On a neighbourhood of a critical point of $f_{\mathrm{sys}}$, $|X(x)|\geq \sinh(d(x,p))$, where $d(x,p)$ is the Weil-Petersson distance in $\mathcal{T}_{g}$ between $x$ and $p$, and $|X(x)|$ is the Weil-Petersson length of the vector.
\end{enumerate}
\end{lem}

The goal will be to use the flow $\phi_{t}$ generated by $X$ to obtain a $\Gamma_{g}$-equivariant deformation retraction of $\mathcal{P}_{g}$ onto a CW-complex $\mathcal{P}_{g}^{X}$, which will be thought of as a piecewise-smooth analogue of the union of unstable manifolds of the critical points of $f_{\mathrm{sys}}$, defined below. Condition (4)  is included to ensure that the retraction occurs in finite time, giving a deformation retraction rather than merely a retraction. The specific choice given in Condition (4) is not essential.

\begin{proof}
Conditions (1) and (3) are possible as a result of Theorem \ref{thm1}. Condition (2) is possible because all quantities and constraints are defined in terms of $\Gamma_{g}$-equivariant quantities. Condition (4) can be achieved by rescaling with the help of a $\Gamma_{g}$-equivariant partition of unity. \\

\begin{figure}[h!]
\centering
\includegraphics[width=0.7\textwidth]{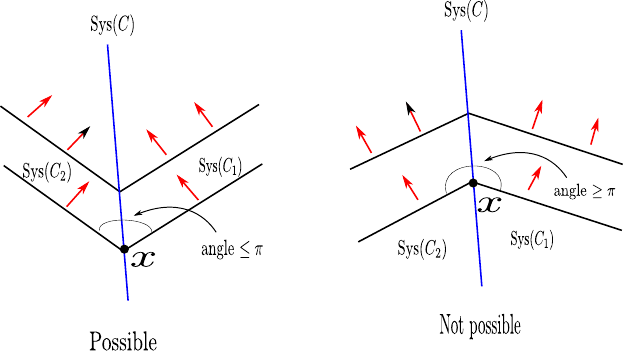}
\caption{Level sets on $\partial I(C,x)$ are shown in black, the stratum $\mathrm{Sys}(C)$ in blue and the red arrows show the direction of $X$. If the flowlines could diverge at $x\in \mathrm{Sys}(C)$ as shown on the right, this would contradict the fact that $I(C,x)$ is contained in a half-space. }
\label{singlevalued}
\end{figure}

It remains to show that the vector field defined this way is well defined. At regular points of $f_{\mathrm{sys}}$, this follows from a property of the cone of increase illustrated in Figure \ref{singlevalued}. Clearly, $X$ is well behaved in the interior of locally top dimensional simplices of $\mathcal{P}_{g}$; it remains to show that $X$ has a unique value at every point of $\mathcal{P}_{g}$ where larger dimensional strata meet along lower dimensional strata. Call one such point in a lower dimensional stratum $x$, and let $C$ be the set of systoles at $x$. Recall that the cone of increase at $x$ is given by the tangent cone at $x$ to $I(C,x)$. By local finiteness, on a neighbourhood of $x$, the level set of $f_{\mathrm{sys}}$ passing through $x$ lies along $\partial I(C,x)$. Level sets have no kinks when restricted to a stratum, so the distinct smooth pieces on $\partial I(C,x)$ determine different strata. Where $X$ is given by the gradient of the restriction of $f_{\mathrm{sys}}$ to a simplex, $X$ is orthogonal to the level sets. By construction, $I(C,x)$ is strictly contained in a half-space, so the level sets of $f_{\mathrm{sys}}$ on $\partial I(C,x)$ near $x$ cannot make more than an angle at $x$ internal to $I(C,x)$ of more than $\pi$. If $X$ were not well defined at $x$, this would give a contradiction by forcing the internal angle in $I(C,x)$ to be too large. In the left hand side of Figure \ref{singlevalued}, at $x$, the flowline is contained in the lower dimensional stratum $\mathrm{Sys}(C)$.\\

This concludes the proof of the lemma away from critical points.\\
\begin{figure}[h!]
\centering
\includegraphics[width=0.5\textwidth]{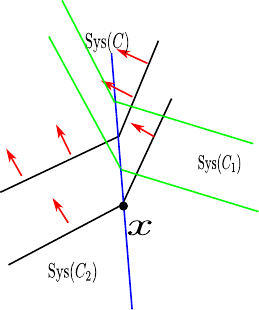}
\caption{Level sets are shown in black. Flowlines, which appear to be refracted by the lower dimensional stratum (blue), are shown in green. }
\label{refraction}
\end{figure}

\begin{rem}
\label{focusorpass}
Away from critical points and boundary points, flowlines either converge on lower dimensional strata of $\mathcal{P}_{g}$, as shown on the left in Figure \ref{singlevalued}, or pass through and are ``refracted'' by lower dimensional strata, as shown in Figure \ref{refraction}.
\end{rem}

That $X$ is well defined at a critical point of $f_{\mathrm{sys}}$ is a consequence of Corollary \ref{cor}; at such points there are no directions where $f_{\mathrm{sys}}$ is increasing to first order and consequently $X=0$.
\end{proof}



Flows defined on simplicial complexes are a standard tool in optimisation and combinatorics, but differ from flows defined on manifolds, as the existence and uniqueness theorem for ODEs is only guaranteed to hold on the interiors of simplices where the vector field is smooth. Due to the fact that $\mathcal{P}_{g}^{X}$ is obtained from a flow on a simplicial complex, rather than on a manifold where the existence and uniqueness theorem for ODEs holds everywhere, the unstable manifolds for distinct critical points might not be disjoint. \\

Suppose $\sigma_{1}$ and $\sigma_{2}$ are a pair of simplices of $\mathcal{P}_{g}$ with a common boundary simplex $\sigma_{3}$. When two flowlines originating in $\sigma_{1}$ and $\sigma_{2}$ approach the same point $x$ in $\sigma_{3}$, the two flowlines will be said to \textit{merge} at $x$, where $x$ is a \textit{merge point}. \\

\textbf{Stable and unstable manifolds.} Stable and unstable manifolds of critical points of topological Morse functions on topological manifolds were defined in \cite{Morse}, but it is important to keep in mind that these might not be unique as in the smooth case. On a topological manifold, an unstable manifold of a critical point of index $j$ is a codimension $j$ ball with the Morse function increasing away from the critical point.\\

Unstable manifolds are defined using the vector field $X$. In Subsection \ref{boundarypoints}, it will be explained how to alter this vector field in the presence of boundary points of $f_{\mathrm{sys}}$ so that the only zeros are at critical points. It is this altered vector field that should be assumed in the next definition.\\ 

An unstable manifold $\mathcal{M}(p)$ in $\mathcal{P}_{g}^{X}$ of a critical point $p$ of $f_{\mathrm{sys}}$ is defined to be the set of all $x\in \mathcal{P}_{g}^{X}$ such that
\begin{itemize}
\item{$x$ is on a flowline originating at $p$, and}
\item{$x$ is not a merge point or above a merge point with a flowline originating at a critical point $p'$, for which $p'$ is the endpoint of a flowline originating at $p$.}
\end{itemize}

\begin{figure}[!thpb]
\centering
\includegraphics[width=0.4\textwidth]{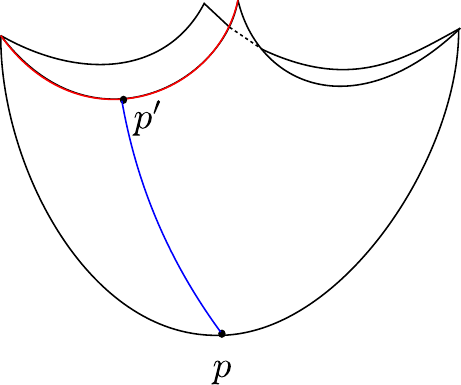}
\caption{An unstable manifold of a critical point $p$, with the unstable manifold of a critical point $p'$ shown in red, and a flowline from $p$ to $p'$ in blue. The Morse function is height.}
\label{maplefigure}
\end{figure}



Recall that the index of the critical point is the codimension of the unstable manifold, so an unstable manifold typically has unstable manifolds of critical points of higher index on its boundary. According to the definition, $\mathcal{M}(p)$ contains only one critical point $p$.

\begin{lem}
\label{flowintersections}
Suppose a critical point $p_{1}$ is the endpoint of a flowline $\gamma$ originating from the critical point $p$. If $\gamma(t_{0})$ is in $\partial \mathcal{M}(p)$ for some $t_{0}\geq 0$, then for all $t>t_{0}$, $\gamma(t)$ is in $\partial \mathcal{M}(p)$.
\end{lem}
\begin{proof}
When working with smooth Morse functions, an analogous statement follows from the observation that solutions to the differential equations defining the flowlines depend continuously on the initial conditions. For $f_{\mathrm{sys}}$, this holds on individual strata making up $\mathcal{M}(p)$, and as discussed in Remark \ref{focusorpass}, the flowlines either pass through or converge on lower dimensional strata.
\end{proof}


\begin{figure}[!thpb]
\centering
\includegraphics[width=0.3\textwidth]{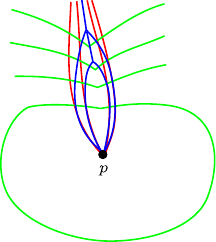}
\caption{Some level sets of $f_{\mathrm{sys}}$ in an unstable manifold $\mathcal{M}(p)$ are shown in green and two pairs of flowlines that each bound a bigon in $\mathcal{M}(p)$ are blue. The altered flowlines are shown in red.}
\label{nobigons}
\end{figure}

It will be assumed that within an unstable manifold $\mathcal{M}(p)$ there are no two flowlines that bound a bigon in $\mathcal{M}(p)$. This can be done without loss of generality by deforming $X$ if necessary within an orbit of unstable manifolds. This alteration can be done equivariantly and involves adding a small component to $X$ tangent to level sets, as illustrated in Figure \ref{nobigons}.\\

\begin{figure}[!thpb]
\centering
\includegraphics[width=0.7\textwidth]{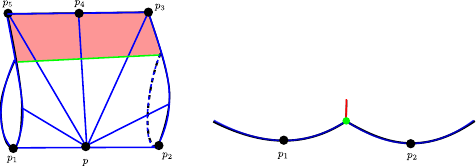}
\caption{Two schematic examples of flaps, shown in red. Flowlines are shown in blue, and merge points in green. Critical points are black dots.}
\label{breakapartflap}
\end{figure}

The next lemma will be assumed from now on without reference.

\begin{lem}
Flowlines can merge, but not split up.
\label{lem14}
\end{lem}
\begin{proof}
The vector field $X$ was defined as the 1-sided limit from above of the tangent to the flowlines. The lemma clearly holds within strata on which $X$ is a gradient. Suppose, as illustrated in Figure \ref{nosplit}, there are two or more strata $\sigma_{1}$ and $\sigma_{2}$ of $\mathcal{P}_{g}$ at $x$ that have the same limiting value of $X$ from above, where this limiting value of $X$ is in the tangent cones to $\sigma_{1}$ and $\sigma_{2}$ at $x$. It remains to show that the flowline through $x$ does not split apart.\\

\begin{figure}[!thpb]
\centering
\includegraphics[width=0.3\textwidth]{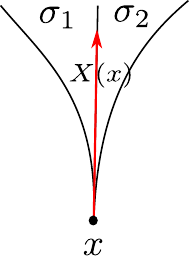}
\caption{Two strata, $\sigma_{1}$ and $\sigma_{2}$ of $\mathcal{P}_{g}$, are shaded in red and blue. At $x$, both these strata have the same limiting value from above of the vector field $X$, and this limiting value is in the tangent cones to both $\sigma_{1}$ and $\sigma_{2}$.}
\label{nosplit}
\end{figure}

Since the point $x$ is not a critical or boundary point of $f_{\mathrm{sys}}$, by Theorems \ref{thm1} and \ref{thm2}, the intersection of the unit cone of increase at $x$ and the tangent cone to $\mathcal{P}_{g}$ at $x$ is connected. Instead of splitting up, flowlines converge on a stratum between $\sigma_{1}$ and $\sigma_{2}$. If there is no stratum between $\sigma_{1}$ and $\sigma_{2}$, the flowlines converge on whichever of $\sigma_{1}$ and $\sigma_{2}$ has the smaller dimension at points directly above $x$.
\end{proof}

A connected component of $\mathcal{M}(p)$ consisting of merge points and the segments of flowlines above the merge points will be called a \textit{flap}. Two examples are shown schematically in Figure \ref{breakapartflap}. Flaps do not contain critical points. Up to homotopy equivalence, the existence of flaps does not change the topology of $\mathcal{P}_{g}^{X}$. There is a homotopy equivalence from $\mathcal{P}_{g}^{X}$ to a complex without flaps, constructed by collapsing to a point the segments of flowlines in the flaps.

\begin{lem}
\label{celllem}
In the absence of flaps, $\mathcal{M}(p)$ is an open cell.
\end{lem}
\begin{proof}
This follows from the same argument as Lemma \ref{flowintersections}.
\end{proof}

Assuming no merge points in the interior of $\mathcal{M}(p)$, an unstable manifold $\mathcal{M}(p)$ of a critical point $p$ of $f_{\mathrm{sys}}$ of index $j$ is a codimension $j$ ball on which $f_{\mathrm{sys}}$ increases away from $p$. By Lemma \ref{flowintersections}, $\mathcal{M}(p)$ has boundary consisting of a union of flowlines on $\mathcal{M}(p)$ in the unstable manifolds of other critical points. Flowlines in $\mathcal{M}(p)$ originate at $p$, and end on a critical point on the boundary of $\mathcal{M}(p)$ or merge with a flowline contained in the boundary of $\mathcal{M}(p)$.

\subsection{Boundary points of the systole function}
\label{boundarypoints}
The purpose of this subsection is to present what is known about the structure of boundary points of the systole function, and to argue that the vector field $X$ can be altered slightly on a neighbourhood of such points to remove the zeros at boundary points, without creating zeros elsewhere.\\

To understand what a boundary point of $f_{\mathrm{sys}}$ is and how it relates to a critical point, it helps to introduce the notion of sets of minima, as defined in \cite{SchmutzMorse}. Choose and fix an ordered set $C$ of curves. Recall that for every fixed $A\in \mathbb{R}_{+}^{C}$ the length function $L(A,C)$ is convex. Lemma 1 of \cite{MorseSmale} shows that if $C$ fills, there is a unique point in $\mathcal{T}_{g}$ at which $L(A,C)$ has its minimum. The set of minima $\mathrm{Min}(C)$ is the set of all such points of $\mathcal{T}_{g}$ at which minima occur, when $A$ ranges over all tuples in $\mathbb{R}_{+}^{C}$.\\

A boundary point of $f_{\mathrm{sys}}$ is a regular point (i.e. noncritical point) at which the cone of increase is not full, i.e. there is no direction in which $f_{\mathrm{sys}}$ is increasing to first order. This is a phenomenon that one does not see with smooth Morse functions. In \cite{SchmutzMorse} it was shown that a boundary point of $f_{\mathrm{sys}}$ with set of systoles $C$ is 
\begin{enumerate}
\item{isolated,}
\item{contained in $\mathcal{P}_{g}$, and}
\item{is a point at which $\mathrm{Sys}(C)$ intersects $\overline{\mathrm{Min}(C)} \setminus \mathrm{Min}(C)$.}
\end{enumerate}
The third property is presumably the reason for the name boundary point of $f_{\mathrm{sys}}$. Property (3) should be compared with the defining property of a critical point with set of systoles $C$, which was shown in \cite{Akrout} to be a point at which $\mathrm{Sys}(C)$ intersects $\mathrm{Min}(C)$.\\

The local structure of $f_{\mathrm{sys}}$ around boundary points was studied in Section 5 of \cite{MorseSmale}. It follows from point (3) above, together with Lemma 14 of \cite{SchmutzMorse}, that $p$ is contained in $\mathrm{Min}(C')$ with $C'\subsetneq C$.\\

The alteration to the vector field $X$ will be motivated by describing a boundary point as a point that arises when two critical points of index differing by 1 are cancelled. Readers interested only in proofs can skip the next two paragraphs. Suppose $\lambda>1$, and note that when the convex function $L(c)$ is replaced by the convex function $\lambda L(c)$ in the definition of $\mathrm{Min}(C)$, the same set of minima is obtained. When $C\setminus C'$ contains a single curve $c$, define $f(\lambda)_{\mathrm{sys}}:\mathcal{T}_{g}\rightarrow \mathbb{R}_{+}$ to be the function whose value at $x$ is given by
\begin{equation*}
\min\{L(c')(x), \lambda L(c)(x)\ |\ c'\neq c\text{ is a curve on }S_{g}\}
\end{equation*}
Since $\lambda L(c)$ is a convex function on $\mathcal{T}_{g}$, one can show that this implies $f(\lambda)_{\mathrm{sys}}$ is also a topological Morse function on $\mathcal{T}_{g}$, see for example \cite{convexMorse}. Theorem 15 of \cite{SchmutzMorse} implies there is a derivation $v$ in the tangent cone to $\mathrm{Min}(C)$ at $p$ with the property that the evaluation of $v$ on $\lambda L(c)$ is negative and the evaluation on $L(c')$ for every $c'\in C'$ is zero. Consequently, for $\lambda$ sufficiently close to $1$, $f(\lambda)_{\mathrm{sys}}$ has two critical points, $p_{1}$ and $p_{2}$. The points $p_{1}$ and $p_{2}$ are close together, their indexes differ by 1, and there is a flowline from $p_{1}$ to $p_{2}$. One of these critical points, call it $p_{1}$, is the same as the point $p\in \mathrm{Min}(C')$, and $p_{2}\in \mathrm{Min}(C)$ (with smaller index) is in $\mathrm{Sys}(\{c\})$ and hence not in $\mathcal{P}_{g}$. In the limit as $\lambda$ approaches 1, the critical points cancel to give the boundary point of $f_{\mathrm{sys}}$. \\

\begin{figure}
\centering
\includegraphics[width=\textwidth]{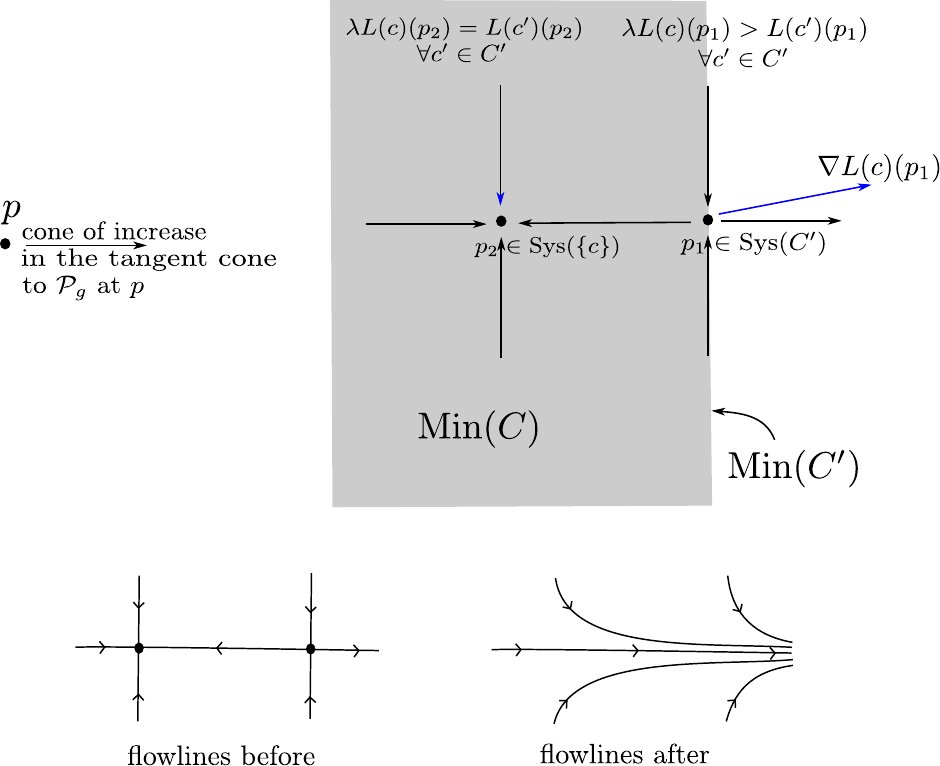}
\caption{The top left figure shows the boundary point $p$ and its cone of increase, which is in the tangent cone to $\mathcal{P}_{g}$ at $p$. To the right of that are the critical points of $f(\lambda)_{\mathrm{sys}}$. Below are the flowlines of $f(\lambda)_{\mathrm{sys}}$ before and after the critical points were cancelled.}
\label{criticalcancellation}
\end{figure}

Theorem 5.4 of \cite{hcobordism} describes the cancellation of the smooth analogues of critical points $p_{1}$ and $p_{2}$ by altering the vector field on a neighbourhood of the flowline between them, as shown in Figure \ref{criticalcancellation}. In local coordinates, the cancellation is achieved by altering the function in such a way that a vector field parallel to the flowline in the direction of $p_{1}$ is added to the gradient at points near the flowline between $p_{1}$ and $p_{2}$. This gives a vector field that agrees with the vector field outside of a neighbourhood of the flowline, and without any zeros within the neighbourhood. The same will now be done to the vector field $X$.\\

Choose $\epsilon>0$ small. Let $\phi$ be a smooth function equal to 1 on the $\frac{\epsilon}{2}$ neighbourhood of $p$, and zero outside of the $\epsilon$-neighbourhood of $p$. Add to $X$ the projection of $\phi (\lambda-1)\nabla\sum_{c\in C\setminus C'}L(c)$ to the tangent cone of $\mathcal{P}_{g}$. \\

\begin{lem}
For sufficiently small $\epsilon$ and sufficiently small $\lambda>1$, the vector field obtained by adding to $X$ the projection of $\phi (\lambda-1)\nabla\sum_{c\in C\setminus C'}L(c)$ to the tangent cone of $\mathcal{P}_{g}$ is nowhere vanishing on the support of $\phi$. 
\end{lem}
\begin{proof}
By Corollary \ref{cor}, the cone of increase at $p$ is contained in the tangent cone to $\mathcal{P}_{g}$. On the strata with tangent cone at $p$ given by the cone of increase at $p$, for small enough $\epsilon$, the altered vector field is nonvanishing on the support of $\phi$. This is because, as shown in Section 5 of \cite{MorseSmale}, the cone of increase at $p$ consists of the vectors in $\{\nabla L(c')(p)\ |\ c'\in C'\}^{\perp}$ with positive inner product with the gradient of the length of every curve in $C\setminus C'$.\\

At $p$, there might also be some strata into which $f_{\mathrm{sys}}$ is decreasing away from $p$ to first order. As long as $\lambda>1$ is close enough to 1 and $\epsilon$ is sufficiently small, the alteration does not create zeros in these strata either.
\end{proof}

The alteration to $X$ can be done in a $\Gamma_{g}$-equivariant way, by using the action of $\Gamma_{g}$ to map the choices to different boundary points in the same orbit. 
\subsection{Unmatched faces}
\label{unmatchedfaces}
The boundary of a simplicial complex consists of a union of unmatched faces, where an unmatched face is a facet of a locally top-dimensional simplex, that is only on the boundary of one locally top-dimensional simplex. As a set, the boundary of the cell complex $\mathcal{P}_{g}$ does not depend on how it is subdivided into a simplicial complex. This is a consequence of the subdivision theorems of Chapter 3.3 of \cite{Spanier}. The purpose of this subsection is to prove the following lemma.

\begin{lem}
\label{nonemptyboundary}
The CW complex $\mathcal{P}_{g}^{X}$ has nonempty boundary whenever $\mathcal{P}_{g}^{X}$ has dimension greater than $4g-5$.
 \end{lem}
\begin{proof}
Let $q$ be a point of $\mathcal{P}_{g}$ in the interior of a locally top-dimensional simplex $\sigma$. By Remark 7 of \cite{Me2023} the pre-image of $q$ under the deformation retraction is a ball $\mathcal{B}(q)$ with empty boundary, intersecting $\mathcal{P}_{g}$ in the single point $q$ and with dimension equal to the codimension of $\sigma$ in $\mathcal{T}_{g}$ at $q$. As $\mathcal{P}_{g}^{X}$, just like $\mathcal{P}_{g}$, is obtained as the image of an equivariant deformation retraction generated by the flow of an $f_{\mathrm{sys}}$-increasing vector field, the argument in Remark 7 of \cite{Me2023} holds with $\mathcal{P}_{g}^{X}$ in place of $\mathcal{P}_{g}$. Consequently, for a point $q$ in the interior of a locally top-dimensional cell of $\mathcal{P}_{g}^{X}$, there is a ball $\mathcal{B}(q)$ obtained as the pre-image of $q$ under the deformation retraction, with codimension equal to the dimension of $\mathcal{P}_{g}^{X}$ at $q$, intersecting $\mathcal{P}_{g}^{X}$ only in the point $q$.\\

The ball $\mathcal{B}(q)$ intersects $\mathcal{T}_{g}^{\epsilon_{M}}$ in a connected set, that separates each flowline in $\mathcal{B}(q)$ into the segment on which $f_{\mathrm{sys}}>\epsilon_{M}$ and the segment on which $f_{\mathrm{sys}}<\epsilon_{M}$. The boundary of $\mathcal{T}_{g}^{\epsilon_{M}}$ is piecewise smooth, and $\mathcal{B}(q)$ intersects each of the top dimensional pieces in such a way that the inward facing normal to each piece makes an angle with the tangent space to $\mathcal{B}(q)$ that is bounded away from $\pi/2$, i.e. $\mathcal{B}(q)$ intersects $\partial\mathcal{T}_{g}^{\epsilon_{M}}$ transversely. The intersection of $\mathcal{B}(q)$ with $\partial\mathcal{T}_{g}^{\epsilon_{M}}$ is therefore a sphere $\mathcal{S}^{\epsilon_{M}}(q)$ of dimension 1 less than the codimension of $\mathcal{M}(p)$ in $\mathcal{T}_{g}$. \\

The dimension of $\mathcal{M}(p)$ cannot be less than $4g-5$, as this is the virtual cohomological dimension of $\Gamma_{g}$, and gives a lower bound on the dimension of any spine. Assume the dimension of $\mathcal{M}(p)$ is greater than $4g-5$. In this case, $\mathcal{S}^{\epsilon_{M}}(q)$ has dimension less than $2g-2$. Since $\partial \mathcal{T}_{g}^{\epsilon_{M}}$ is $\Gamma_{g}$-equivariantly homotopy equivalent to a wedge of spheres $\vee^{\infty}S^{2g-2}$, \cite{Ivanov2002}, $\mathcal{S}^{\epsilon_{M}}(q)$ is contractible in $\partial \mathcal{T}_{g}^{\epsilon_{M}}$. Moreover $\mathcal{B}(q)\cap \mathcal{T}_{g}^{\epsilon_{M}}$ can be homotoped relative to its boundary into $\partial\mathcal{T}_{g}^{\epsilon_{M}}$.\\

The homotopy of $\mathcal{B}(q)\cap \mathcal{T}_{g}^{\epsilon_{M}}$ into $\partial\mathcal{T}_{g}^{\epsilon_{M}}$ moves the point $q$ off $\mathcal{P}_{g}^{X}$. This implies that $\mathcal{P}_{g}^{X}$ has nonempty boundary, as required.
\end{proof}

\subsection{Orbits of elementary deformation retractions}
\label{construction}
This subsection uses the topology of the boundary of the thick part of $\mathcal{T}_{g}$ and the results of the previous subsections to prove the next theorem.

\begin{thm}
There is a mapping class group-equivariant deformation retraction of the Thurston spine of a closed orientable surface of genus $g$ onto a CW-complex of dimension equal to $4g-5$.
\end{thm}
\begin{proof}
The theorem will first be proven under the assumption that there are no flaps, and then generalised.\\

It is necessary to understand the way fixed point sets of $\Gamma_{g}$ intersect the level sets within unstable manifolds of critical points, to show that they do not obstruct the construction of a $\Gamma_{g}$-equivariant deformation retraction. A fundamental domain of the action of $\Gamma_{g}$ on $\mathcal{M}(p)$ will be chosen.\\

Denote by $G(p)$ the subgroup of $\Gamma_{g}$ that stabilises $\mathcal{M}(p)$ as a set, not necessarily pointwise. The group $G(p)$ is finite as a consequence of Fricke's theorem. \\

By Lemma \ref{celllem}, in the absence of flaps, $\mathcal{M}(p)$ is an open cell. By Proposition \ref{fixedpointsetprop}, all fixed point sets intersect the unstable manifold in a critical point or are transverse to the level sets both of $\mathcal{M}(p)$ and of the unstable manifolds on the boundary of $\mathcal{M}(p)$. \\

\begin{figure}[!thpb]
\centering
\includegraphics[width=0.4\textwidth]{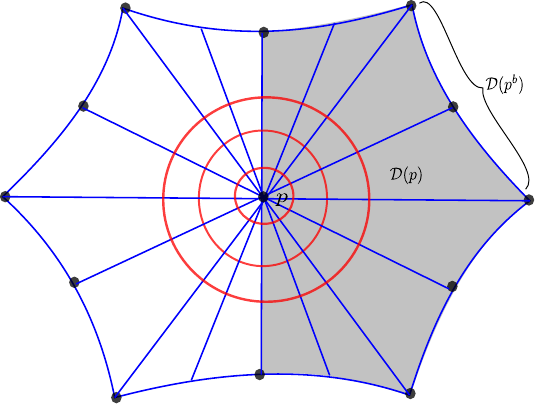}
\caption{A fundamental domain (shaded) of the action of $G(p)$ on the closure of $\mathcal{M}(p)$. Flowlines are shown in blue, level sets in red, and dots are critical points.}
\label{fundamentaldomain}
\end{figure}

\textbf{Constructing fundamental domains.} As the flow used to construct $\mathcal{P}_{g}^{X}$ is $\Gamma_{g}$-equivariant, $G(p)$ maps flowlines to flowlines. Moreover, $G(p)$ maps critical points to critical points of the same index, unstable manifolds to unstable manifolds, level sets to level sets, and fixes $p$. The fixed point set of any subgroup of $G(p)$ therefore contains $p$. \\

If a flowline $\gamma(t)$ is tangent to a fixed point set at time $t_{0}$, it is tangent at all times $t>t_{0}$ in its domain, otherwise the flow would not be $\Gamma_{g}$-equivariant, or would violate Lemma \ref{lem14}. This same argument shows that the assumption of no flaps ensures that the only fixed point sets of $\Gamma_{g}$ that intersect the interior of $\mathcal{M}(p)$ are (pointwise) fixed point sets of subgroups of $G(p)$. A fundamental domain of the action of $G(p)$ on the closure of $\mathcal{M}(p)$ can therefore be chosen to be invariant under the flow on $\mathcal{M}(p)$. The boundary of $\mathcal{M}(p)$ is contained in unstable manifolds of critical points distinct from $p$. The assumption of no flaps ensures that by Lemma \ref{flowintersections}, every unstable manifold $\mathcal{M}(p_{1})$ containing boundary points of $\mathcal{M}(p)$ either is contained in $\partial\mathcal{M}(p)$, or $\partial\mathcal{M}(p)\cap \mathcal{M}(p_{1})$ consists of a union of flowlines and the critical point $p_{1}$.\\

Denote by $p^{b}$ a critical point on $\partial \mathcal{M}(p)$ for which $\mathcal{M}(p^{b})$ contains boundary points of $\mathcal{P}_{g}^{X}$ on $\partial \mathcal{M}(p)$.\\

To make the deformation retraction $\Gamma_{g}$-equivariant, all deformation retractions coming from $\Gamma_{g}$-orbits of $\mathcal{M}(p^{b})$ will be performed simultaneously. \\


When more than one element of the $\Gamma_{g}$-orbit of $\mathcal{M}(p^{b})$ contains an unmatched face of $\mathcal{M}(p)$, there is a nontrivial finite subgroup $G(p)$ of $\Gamma_{g}$ that fixes $\mathcal{M}(p)$. A fundamental domain $\mathcal{D}(p)$ of the action of $G(p)$ on $\mathcal{M}(p)$ is constructed as explained above. A fundamental domain of $\mathcal{M}(p^{b})\cap \partial\mathcal{M}(p)$ under the action of the subgroup of $G(p)$ that stabilises $\mathcal{M}(p^{b})\cap \mathcal{M}(p)$ will be denoted by $\mathcal{D}(p^{b})$. Here it is assumed that $\mathcal{D}(p^{b})$ is on the boundary of $\mathcal{D}(p)$. \\

To construct a $\Gamma_{g}$-equivariant deformation retraction, it suffices to consider the following two cases:
\begin{enumerate}
\item{Every point in the interior of $\mathcal{D}(p^{b})$ is on $\partial\mathcal{P}_{g}^{X}$}
\item{Every point in the interior of $\mathcal{D}(p^{b})$ is on $\partial\mathcal{P}_{g}^{X}$, except for some flowlines on the boundary of a finite number of unstable manifolds of dimension lower than $\mathcal{M}(p)$.}
\end{enumerate}
That these two cases suffice follows from the fact that there are only finitely many critical points modulo the action of $\Gamma_{g}$, from the local model of the neighbourhood of a critical point given in Corollary \ref{cor} and from Lemma \ref{flowintersections}.\\

\textbf{Case 1.} How to construct the retraction in Case 1 will now be explained. This is illustrated schematically in Figure \ref{fundamentaldomain}.\\

Claim - A $\Gamma_{g}$-equivariant deformation retraction of $\mathcal{P}_{g}^{X}$ is obtained by deleting the $\Gamma_{g}$-orbits of the interiors of each of $\mathcal{D}(p^{b})$ and $\mathcal{D}(p)$.\\

To prove the claim, there are two issues to address. First of all, the possibility that the stabiliser subgroup $G(p^{b})$ of $\mathcal{M}(p^{b})\cap \partial \mathcal{M}(p)$ is not contained in $G(p)$ was deliberately ignored. If $G(p^{b})$ were not contained in $G(p)$, the point $q$ of the ball $\mathcal{B}(q)$ could not have been homotoped off $\mathcal{P}_{g}$ through $\mathcal{M}(p^{b})$, because $\mathcal{M}(p^{b})$ would then be on the boundary of both $\mathcal{M}(p)$ and the image of $\mathcal{M}(p)$ under $G(p^{b})$.\\

Secondly, within $\mathcal{D}(p)$, the deformation retraction is achieved by first collapsing the flowlines in $\mathcal{D}(p)$ with endpoints on $\mathcal{D}(p^{b})$, retracting along level sets, and then collapsing any remaining pieces of flowlines, as shown in Figure \ref{fundamentaldomain}. This concludes the proof of the claim.\\

\begin{figure}[!thpb]
\centering
\includegraphics[width=\textwidth]{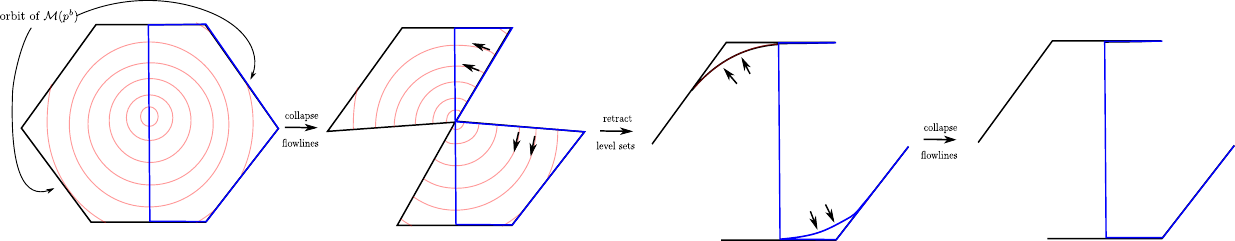}
\caption{To construct the deformation retraction, collapse the flowlines in $\mathcal{D}(p)$ with endpoints on $\mathcal{D}(p^{b})$, retract the level sets, then collapse any remaining flowlines. On the left, a fundamental domain is outlined in blue, and the level sets in red. The figures on the right show the steps of the deformation retraction.}
\label{fundamentaldomain}
\end{figure}

\textbf{Case 2.} Choose a ``wedge'' within $\mathcal{D}(p^{b})$ as illustrated in Figure \ref{wedgechoice}. This is a choice of connected component that is obtained when $\mathcal{D}(p^{b})$ is cut along the boundaries of lower dimensional unstable manifolds. A wedge is foliated by level sets of $f_{\mathrm{sys}}$, and by Lemma \ref{flowintersections} the boundary of the wedge is a union of flowlines.\\

\begin{figure}[!thpb]
\centering
\includegraphics[width=0.4\textwidth]{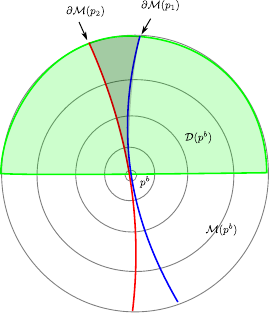}
\caption{In Case 2, $\mathcal{M}(p^b)$ is shown with the boundaries of two lower dimensional unstable manifolds $\mathcal{M}(p_{1})$ and $\mathcal{M}(p_{2})$ shown in blue and red. $\mathcal{D}(p^{b})$ is shaded in green, and a choice of ``wedge'', i.e. a connected component of $\mathcal{D}(p^{b})$ in the complement of $\overline{\mathcal{M}(p_{1})}$ and $\overline{\mathcal{M}(p_{2})}$ is shaded in darker green. }
\label{wedgechoice}
\end{figure}

Claim - A $\Gamma_{g}$-equivariant deformation retraction of $\mathcal{P}_{g}^{X}$ is obtained by deleting the $\Gamma_{g}$-orbit of the wedge and the $\Gamma_{g}$-orbit of the interior of $\mathcal{D}(p)$.\\

To prove the claim, construct a deformation retraction by retracting the level sets of $f_{\mathrm{sys}}$ in $\mathcal{D}(p)$ adjacent to the level sets in the wedge. Then collapse any remaining subintervals of flowlines in $\mathcal{D}(p)$ that passed through the retracted level sets. Once again retract any remaining portions of level sets of $f_{\mathrm{sys}}$ in $\mathcal{D}(p)$, and collapse any remaining subintervals of flowlines in $\mathcal{D}(p)$ that passed through the retracted level sets.\\

This concludes the proof of the theorem in the absence of flaps.\\

\textbf{Case with flaps.} The construction of a $\Gamma_{g}$-equivariant deformation retraction of $\mathcal{P}_{g}^{X}$ in the presence of flaps will now be discussed.\\

The presence of flaps allows for the possibility of a finite subgroup or subgroups of $\Gamma_{g}$, not contained in $G(p)$, whose fixed point set intersects the interior of $\mathcal{M}(p)$. These are fixed point sets of groups that permute the unstable manifolds that meet along a flap, as illustrated in Figure \ref{twistflap}. As before, if a flowline $\gamma(t)$ is tangent to one such fixed point set at time $t_{0}$, it is tangent at all times $t>t_{0}$ in its domain. When $t_{0}$ with this property is chosen as small as possible, $\gamma(t_{0})$ is a merge point, not a critical point. \\

As above, the fact that $\Gamma_{g}$ maps flowlines to flowlines  makes it possible to construct a fundamental domain of the action of $G(p)$ on $\mathcal{M}(p)$ with the property that the fundamental domain is a connected component obtained by cutting $\mathcal{M}(p)$ along flowlines. \\

\begin{figure}[!thpb]
\centering
\includegraphics[width=0.4\textwidth]{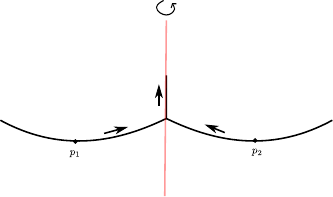}
\caption{Suppose $\mathcal{M}(p_{1})$ and $\mathcal{M}(p_{2})$ intersect along a flap, as shown. It is possible that there is a finite subgroup $G(p_{1},p_{2})$ of $\Gamma_{g}$ that  interchanges the critical points $p_{1}$ and $p_{2}$. In this case, the fixed point set of $G(p_{1}, p_{2})$ shown in red can intersect both $\mathcal{M}(p_{1})$ and $\mathcal{M}(p_{2})$, even though $G(p_{1},p_{2})$ is not contained in $G(p_{1})$ or $G(p_{2})$.}
\label{twistflap}
\end{figure}

When constructing the deformation retraction of the $\Gamma_{g}$-orbit of $\mathcal{M}(p)$, first construct the fundamental domain $\mathcal{D}(p)$ as above. The construction is then broken up into the following cases.

\begin{enumerate}
\item{$\mathcal{D}(p^{b})$ is on $\partial \mathcal{P}_{g}^{X}$ except for the points that are merge points of flaps. This is illustrated in Figure \ref{simpleflap}.}
\item{$\mathcal{D}(p^{b})$ is partially on $\partial \mathcal{P}_{g}^{X}$ and partially on $\partial \mathcal{M}(p_{2})$ where the dimension of $\mathcal{M}(p_{2})$ is equal to the dimension of $\mathcal{M}(p)$. This is illustrated in Figure \ref{doubleflap}.}
\item{$\mathcal{D}(p^{b})$ is only on $\partial \mathcal{P}_{g}^{X}$ where it lies on the boundary of a flap or flaps. This is illustrated in Figure \ref{tent}.}
\item{$\mathcal{D}(p^{b})$ is on $\partial \mathcal{P}_{g}^{X}$, both on the boundary of a flap and on $\mathcal{M}(p)\setminus \text{flaps}$. This is illustrated in Figure \ref{rolledup}}
\item{A combination of the previous case with cases 1 and 2 from the flapless discussion above.}
\end{enumerate}

\textbf{Case 1.} The image of the deformation retraction is obtained by deleting the $\Gamma_{g}$-orbit of $\mathcal{D}(p^{b})\setminus \text{flaps}$ and the $\Gamma_{g}$-orbit of the interior of $\mathcal{D}(p)$ from $\mathcal{P}_{g}^{X}$. As illustrated in Figure \ref{simpleflap}, this deformation retraction is achieved by generalising the steps from Case 1 above.\\

\begin{figure}
\centering
\includegraphics[width=0.8\textwidth]{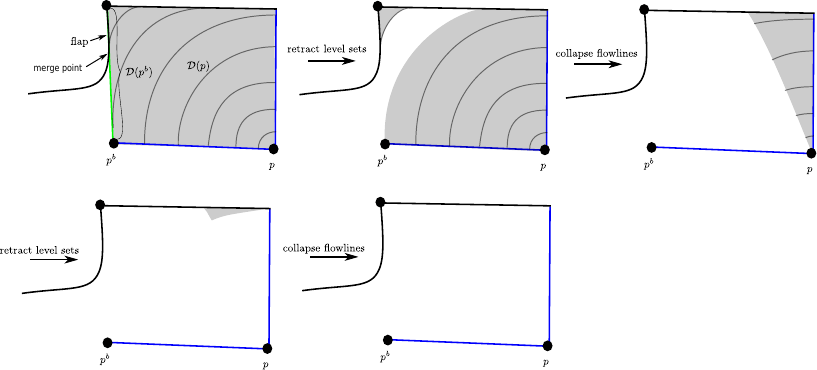}
\caption{The construction of the deformation retraction in Case 1. The boundary of $\mathcal{P}_{g}^{X}$ is shown in green.}
\label{simpleflap}
\end{figure}

\textbf{Case 2.} In this case, the point $q$ could just as well have been homotoped off $\mathcal{P}_{g}^{X}$ through the boundary of an unstable manifold near $\mathcal{M}(p)$. This unstable manifold is denoted by $\mathcal{M}(p_{2})$ in Figure \ref{doubleflap}. To construct the deformation retraction, consider $\mathcal{M}(p_{2}^{b})$ and $\mathcal{M}(p_{2})$ in place of $\mathcal{M}(p^{b})$ and $\mathcal{M}(p)$.\\

\begin{figure}[!thpb]
\centering
\includegraphics[width=0.6\textwidth]{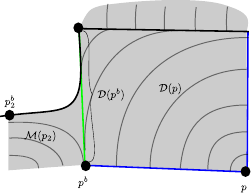}
\caption{An unstable manifold $\mathcal{M}(p_{2})$ adjacent to $\mathcal{M}(p)$ in Case 2. }
\label{doubleflap}
\end{figure}

\textbf{Case 3.} In this case, the point $q$ was chosen to be in a flap of the type depicted in Figure \ref{tent}. It is possible to perform a deformation retraction collapsing the subintervals of flowlines in the intersection of the flap. This does not reduce the dimension of $\mathcal{P}_{g}^{X}$, so if $\mathcal{P}_{g}^{X}$ has dimension greater than $4g-5$, there must be a choice for $q$ that does not give this case. \\

Whenever a deformation retraction creates flowlines whose initial point is not on a critical point, a deformation retraction is performed by collapsing sub-intervals of flowlines above the initial points. This preserves the property that the dimensions of any flaps are no larger than the dimensions of the remnants of the unstable manifolds in which they are contained. \\

\begin{figure}[!thpb]
\centering
\includegraphics[width=0.6\textwidth]{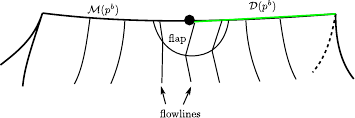}
\caption{An example in which $\mathcal{D}(p^{b})$ (shown in green) is only on $\partial \mathcal{P}_{g}^{X}$ where it lies on the boundary of a flap or flaps.}
\label{tent}
\end{figure}

\textbf{Case 4.} In this case, a flap can occur where flowlines within the same unstable manifold merge, as illustrated in Figure \ref{rolledup}. The unstable manifold might not be contractible. The image of the deformation retraction is obtained by deleting the $\Gamma_{g}$-orbit of $\mathcal{D}(p^{b})\setminus \text{flaps}$ and the $\Gamma_{g}$-orbit of the interior of $\mathcal{D}(p)\setminus \text{flaps}$ from $\mathcal{P}_{g}^{X}$.  As illustrated in Figure \ref{rolledup}, this deformation retraction is achieved by repeatedly retracting level sets and collapsing segments of flowlines.\\

\begin{figure}[!thpb]
\centering
\includegraphics[width=\textwidth]{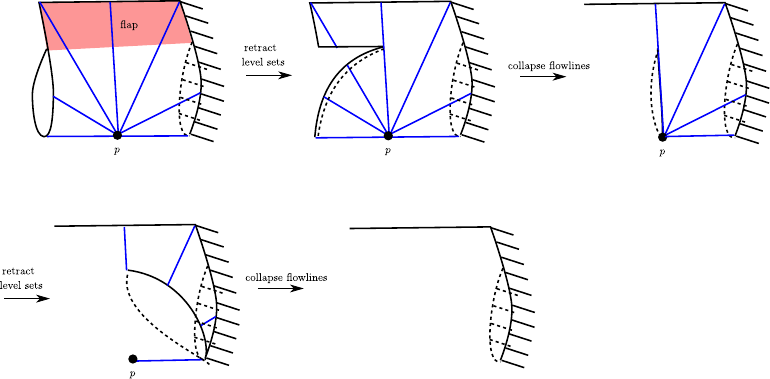}
\caption{When flowlines within $\mathcal{M}(p)$ merge, $\mathcal{M}(p)$ is not contractible. A schematic representation of the deformation retraction in Case 4 is shown.}
\label{rolledup}
\end{figure}

\textbf{Case 5.} In this case, the deformation retractions are obtained by generalising in straightforward ways the constructions described for the previous cases.\\

There is some choice involved in the construction of $\mathcal{D}(p)$. Boundary cells of $\mathcal{D}(p)$ not lying completely along fixed point sets of finite subgroups of $\Gamma_{g}$ are not uniquely determined. Any such boundary cells that remain after the deformation retraction are not stabilised by any subgroups of $\Gamma_{g}$. As it will not be necessary to worry about subdividing these into fundamental domains in later steps, the choices are not important. Moreover, each of the deformation retractions described above preserves the properties of flowlines and level sets required for the construction. This argument can therefore be iterated. The iterations must terminate after finitely many steps, as there are only finitely many orbits of critical points and strata.
\end{proof}

\bibliography{Spinebib}
\bibliographystyle{plain}

\end{document}